\documentclass[11pt]{amsart}

\newtheorem{theorem}{Theorem}[section]
\newtheorem{lem}[theorem]{Lemma}
\newtheorem{prop}[theorem]{Proposition}

\theoremstyle{definition}

\theoremstyle{remark}

\numberwithin{equation}{section}

\newcommand{\ra}{\rightarrow}

\newcommand{\HH}{{\mathbb{H}}}

\newcommand{\HP}{{\mathbb{HP}}}

\newcommand{\ii}{\mathbf{i}}\newcommand{\jj}{\mathbf{j}}\newcommand{\kk}{\mathbf{k}}

\newcommand{\Ad}{\text{Ad}}

\newcommand{\lb}{\left\langle}
\newcommand{\rb}{\right\rangle}

\newcommand{\mh}{\mathfrak{h}}
\newcommand{\mk}{\mathfrak{k}}

\newcommand{\mpp}{\mathfrak{p}}

\newcommand{\R}{\mathbb{R}}

\newcommand{\hx}{\hat{X}}
\newcommand{\hy}{\hat{Y}}
\newcommand{\hz}{\hat{Z}}
\newcommand{\ha}{\hat{\alpha}}
\newcommand{\mC}{\mathcal{C}}
\newcommand{\mF}{\mathcal{F}}
\newcommand{\mQ}{\mathcal{Q}}
\newcommand{\tU}{\mathcal{U}}

\begin{document}

\title[Radially symmetric connections]{Radially symmetric connections over round spheres}

\author{Kristopher Tapp}
\address{Department of Mathematics\\ Saint Joseph's University, 5600 City Ave.\\ Philadelphia, PA 19131}
\email{ktapp@sju.edu}

\begin{abstract}
We classify the radially symmetric connections in vector bundles over round spheres by proving that they are all parallel.
\end{abstract}
\maketitle
\section{Introduction}
Suppose that $B$ is a compact Riemannian manifold and $E$ is the total space of a rank $k$ vector bundle over $B$ endowed with a Euclidean structure (a smoothly varying inner product on the fibers) and a compatible connection $\nabla$.  This connection is called \emph{parallel} if its curvature tensor, $R^\nabla$, has zero covariant derivative, that is, $(D_ZR^\nabla)(X,Y)V=0$ for all $p\in B$, $X,Y,Z\in T_p B$ and $V\in E_p$ = the fiber over $p$.  It is called \emph{radially symmetric} if the $Z=X$ case of this hypothesis is satisfied; that is, $(D_XR^\nabla)(X,Y)V=0$ for all $p\in B$, $X,Y\in T_p B$ and $V\in E_p$.  Notice that both properties depend on the metric of $B$.

Strake and Walshap proved in~\cite{Wal} that when $B$ has positive sectional curvature, radial symmetry is a sufficient condition for $\nabla$ to induce a connection metric with nonnegative sectional curvature on $E$.  Other conditions have appeared in the literature, some necessary and some sufficient, for $\nabla$ to induce a connection metric with nonnegative curvature on $E$ (or of positive or nonnegative curvature on the unit sphere bundle).  All of these conditions boil down to a bound on the norm of $(D_XR^\nabla)(X,Y)V$; see for example~\cite{CDR},\cite{Wal},\cite{T}, and~\cite{Yang}.  This motivates a study of connections satisfying such bounds, and a natural starting point is the radial symmetry condition.

Partly motivated by the application to nonnegative curvature, Guijarro, Lorenzo and Walshap classified the \emph{parallel} connections in vector bundles over irreducible simply connected symmetric spaces: every one is an associated vector bundle of a canonical principal bundle with the connection inherited from the principal bundle~\cite{GSW}.

For rank $2$ vector bundles, parallel is equivalent to radial symmetry (even locally), as observed in~\cite[Remark 6.3]{STT}.  For higher rank, nothing is known about the gap between these two hypotheses.  For vector bundles over the round sphere, we will prove that a condition even weaker than radial symmetry implies parallel:
\begin{theorem}[Main Theorem] Suppose that $E$ is the total space of a vector bundle over $(S^n,\text{round})$, endowed with a Euclidean structure and a compatible connection $\nabla$.  Let $p_0\in S^n$ and let $\ha$ denote the gradient of the distance function to $p_0$, regarded as a unit vector field on $\tU=S^n-\{\pm p_0\}$.  If $(D_{\ha}R^\nabla)(\ha,X)V=0$ for all $p\in\tU$, $X\in T_p S^n$ and $V\in E_p$, then $\nabla$ is parallel.
\end{theorem}

That is, if at each point $\nabla$ satisfies the radial symmetry condition in the ``longitude'' direction, then it must be parallel.

The remainder of the paper is organized as follows.  In Section 2, we present a geometric method for constructing the clutching map, trivializing the bundle locally, and describing the connection form with respect to this trivialization.  This description is nonstandard, but in Section 3 we show how natural it looks for known examples of parallel connections.  In Section 4, we derive curvature formulas for a connection described in this manner.  In Section 5, we prove the main theorem as follows.  The radial symmetry hypothesis rigidly determines the connection form in terms of the clutching map along each longitude geodesic from $p_0$ to $-p_0$.  For the connection to be smooth at $\pm p_0$, these determiniations along individual longitudes must match up in a way that places geometric restrictions on the clutching map, from which the theorem is proven.

\section{Set Up}
Throughout this paper, $E$ will denote the total space of a vector bundle over $(S^n,\text{round})$, endowed with a Euclidean structure and a compatible connection $\nabla$.  In this section, we exhibit an advantageous method of describing $\nabla$ in local coordinates.

First we will express an arbitrary element $p\in S^n$ as
$$p=((\cos\alpha)a,\sin\alpha)\in\R^n\times\R\cong\R^{n+1},$$
where $\alpha\in[-\pi/2,\pi/2]$ and $a\in S^{n-1}\subset\R^n$.  We will sometimes identify $a$ with the element $(a,0)$ of the ``equatorial'' $S^{n-1}$ in $S^n$.  This description of $S^n$ distinguishes the axis between $p^-=(0,-1)$ and $p^+=(0,1)$, which is an arbitrary antipodal pair, since we are free to chose the coordinates of $\R^{n+1}$.  Notice that $S^n$ is covered by the following pair of contractible subsets:
$$U^+=S^n-\{p^-\}\,\text{ and }\,U^-=S^n-\{p^+\},$$
named so that $p^\pm$ is the center of $U^\pm$.  Define $\tU=U^+\cap U^-$.
We will sometimes consider $a:\tU\ra  S^n$ to be the function sending $((\cos\alpha)a,\sin\alpha)\mapsto a$ and consider $\alpha:S^n\ra[-\pi/2,\pi/2]$ to be the function that sends $((\cos\alpha)a,\sin\alpha)\mapsto\alpha$.

There is a natural geometric way to trivialize the bundle over $U^\pm$.  First choose arbitrary orthogonal identifications $\R^k\cong E_{p^\pm}$.  Then for each $p\in U^\pm$ and each $v\in\R^k\cong E_{p^\pm}$, let $(p,v)^\pm\in E_p$ denote parallel transport of $v$ along the unique minimizing geodesic from $p^\pm$ to $p$.  The resulting diffeomorphism $U^\pm\times\R^k\ra E$, that sends $(p,v)\mapsto (p,v)^\pm$, is a trivialization of the bundle over $U^\pm$.  The whole bundle can therefore be described as:
$$E\cong\left(\left(U^+\times\R^k\right)\sqcup\left(U^-\times\R^k\right)\right)/\sim$$
where ``$\sim$'' denotes the equivalence relation on the disjoint union defined so that for all $p=((\cos\alpha)a,\sin\alpha)\in \tU$ and all $v\in\R^k$ we have:
\begin{equation}\label{E:clutch}(p,v)^+\sim\left(p,\mC(a)\cdot v\right)^-.\end{equation}
Here $\mC:S^{n-1}\ra SO(k)$ is the \emph{clutching map}, defined so that for $a\in S^{n-1}$, left multiplication by $\mC(a)$ equals the orthogonal transformation from $\R^k\cong E_{p^+}$ to $\R^k\cong E_{p^-}$ obtained by parallel transport along the unique minimizing geodesic from $p^+$ to $p^-$ that passes through $(a,0)$.  The homotopy type of $\mC$ determines the isomorphism class of the vector bundle, while the geometry of $\mC$ encodes information about $\nabla$.

Fix $a_0\in S^{n-1}$, and henceforth assume that the above-mentioned identifications $\R^k\cong E_{p^-}$ and $\R^k\cong E_{p^+}$ are chosen to be related by: $\mC(a_0)=I$ (the identity).  This ensures that for every $a\in S^{n-1}$, $\mC(a)$ equals the parallel transport around the loop $\beta_a$ in $S^n$ that first follows the geodesic segment from $p^+$ to $p^-$ passing through $a$ and then follows the geodesic segment from $p^-$ to $p^+$ passing through $a_0$.  In summary:
\begin{equation}\label{PE}\mC(a) = P_{\beta_a}\text{  = the parallel transport along $\beta_a$}.\end{equation}

For each $v\in\R^k$, let $v^\pm$ denote the corresponding ``constant'' section over $U^\pm$, defined as $v^\pm(p) = (p,v)^\pm$.  The bundle trivializations determine how a $k$-by-$k$ matrix $M$ acts from the left on a section over $U^\pm$, namely:
$$M\cdot v^\pm = (M\cdot v)^\pm\quad\text{ for all }v\in\R^k,$$
with $v$ considered as a column matrix.  But this multiplication is only individually defined over $U^+$ and $U^-$; it does not determine a well-defined operation on all of $E$.

The connection $\nabla$ is represented over $U^\pm$ by a $1$-form, $\omega^\pm$, on $U^\pm$ with values in $ so(k)$, defined so that for every vector field $X$ on $U^\pm$ and every $v\in\R^k$,
\begin{equation}\label{Eq:condef}\nabla_X(v^\pm) =\omega^\pm(X)\cdot v^\pm.\end{equation}
The following Liebniz Rule holds:
\begin{lem}\label{L:lieb} If $X$ is a tangent field on $U^\pm$, $M$ is a smooth function on $U^\pm$ with values in the space of $k$-by-$k$ real matrices, and $v\in\R^k$, then
$$\nabla_X(M\cdot v^\pm) =\left(X(M) + \omega^\pm(X)\cdot M\right)\cdot v^\pm.$$
\end{lem}
\begin{proof} Let $\{e_1,...,e_k\}$ denote the standard orthonormal basis of $\R^k$.  It will suffice to prove the statement when $v=e_s$ for arbitrary fixed $s\in\{1,...,k\}$.  Using subscripts to denote the components of vectors and matrices, we have:
\begin{align*}
\nabla_X(M\cdot e_s^\pm)
   = \nabla_X\left(\sum_i M_{is}\,e_i^+\right)
  & = \sum_{i}X\left(M_{is}\right)\,e_i^\pm + \sum_{i} M_{is}\,\omega^\pm(X)\cdot e_i^\pm \\
  & = \sum_{i}X\left(M_{is}\right)\,e_i^\pm + \sum_{i,j} M_{is}(\omega^\pm(X))_{ji}\,e_j^\pm \\
  & = X(M)\cdot e_s^\pm + \omega^\pm(X)\cdot M\cdot e_s^\pm.
\end{align*}
\end{proof}

We will next describe $\omega^\pm$ in terms of its action on certain Killing fields.  To describe these Killing fields, let $H$ be a Lie group with a transitive left action on the equatorial $S^{n-1}$, let $K$ denotes the isotropy group at $a_0$, and let $\pi:H\ra H/K\cong S^{n-1}$ denote the quotient map.  Let $\mk\subset\mh$ denote the Lie algebras and let $\mathfrak{p}$ denote the orthogonal compliment of $\mk$ in $\mh$, which is naturally identified with $T_{a_0}S^{n-1}$.  If $a\in S^{n-1}$, we denote $\mk_a=\Ad_{g^{-1}}(\mk)$ and $\mpp_a=\Ad_{g^{-1}}(\mpp)$, which does not depend on the choice of $g\in\pi^{-1}(a)$.  Notice that $\mh=\mk_a\oplus\mpp_a$, and $\mpp_a$ is naturally identified with $T_a S^{n-1}$.

Let $\hat\alpha$ denote the gradient of $\alpha$, considered as a unit-length vector field on $\tU$.  For each $X\in\mh$, let $X_R$ denote the associated right-invariant field on $H$, let $\tilde{X}$ denote the associated Killing field on $S^{n-1}$, and let $\hat{X}$ denote the associated Killing field on $S^n$ coming from the left action of $H$ on $S^n$ with fixed points $\{p^+,p^-\}$.

To be more explicit, for $\alpha\in(-\pi/2,\pi/2)$ and $g\in H$, we set $a=\pi(g)\in S^{n-1}$ and $p=((\cos\alpha)a,\sin\alpha)\in S^n$ and observe that:
$$\hat\alpha(p)=\left(-(\sin\alpha)a,\cos\alpha\right),\quad X_R(g) = X\cdot g,\quad
  \tilde{X}(a)=\pi_*(X_R(g)),\quad\hat{X}(p) = ((\cos\alpha)\tilde{X}(a),0).$$
Notice that $\omega^\pm(\ha)=0$ because we trivialized the bundle via parallel transport along longitude geodesics.  Therefore $\omega^\pm$ can be fully described in terms of the smooth function that assigns to each $p\in U^\pm$ the linear map $\mF^\pm_p:\mh\ra so(k)$ defined as:
\begin{equation}\label{E:dud}\mF^\pm_p(X) = \omega^\pm_p(\hat{X}).\end{equation}
Notice that $\mF^\pm_p$ vanishes on $\mk_a$, and is therefore determined by its values on $\mpp_a$.

The function $\mF^\pm$ was defined by Equations~\ref{Eq:condef} and~\ref{E:dud}, but it could alternatively be described purely in terms of parallel transport.  If $p=((\cos\alpha)a,\sin\alpha)\in\tU$ and $X\in\mpp_a$, we can compute $\mF^\pm_p(X)$ by considering the loop $\gamma_t$ that traverses the minimal geodesic from $p^\pm$ to $p$, followed by the flow for time $t$ along $\hx$, followed by the minimal geodesic back to $p^\pm$.  It is straightforward to see that:
\begin{equation}\label{beta}
\mF^\pm_p(X) = -\frac{d}{dt}\Big|_{t=0} P_{\gamma_t}.\end{equation}

We next define $\mQ$ as the function that associates to each $a\in S^{n-1}$ the linear map $\mQ_a:\mh\ra so(k)$ defined so that for all $X\in\mh$:
$$\mQ_a(X) = d\mC_{a}(\tilde{X}(a))\cdot \mC(a)^{-1}.$$
Notice that $\mQ_a$ vanishes on $\mk_a$ and is therefore determined by its values on $\mpp_a$.  The continuity of the connection at $p^+$ (the point that's missing from $U^-$) implies the following asymptotic behavior of $\mF^-$:

\begin{prop}\label{Mags}Fix $a\in S^{n-1}$ and define $\gamma:\left(-\frac{\pi}{2},\frac{\pi}{2}\right)\ra S^n$ as $\gamma(t) = \left((\cos t)a,\sin t\right)$.  For any $X\in \mh$,
$$\lim_{t\ra\frac{\pi}{2}} \mF^-_{\gamma(t)}(X) = -\mQ_a(X).$$
\end{prop}

\begin{proof}
Since both sides of the desired equation vanish when $X\in\mk_a$, it will suffices to prove the equation when $X$ is a unit-length element of $\mpp_a$, in which case $|\hx(\gamma(t))|=\cos t$.  The following is a smooth unit-length tangent field along $\gamma$:
$$X(t) = \frac{\hat{X}(\gamma(t))}{|\hat{X}(\gamma(t))|} = \frac{\hat{X}(\gamma(t))}{\cos(t)}.$$
Fix $v\in\R^k$ and consider the section $W(t)  =  \nabla_{X(t)}\left(v^+\right)$ along $\gamma$.  Equation~\ref{E:clutch} and Lemma~\ref{L:lieb} yield:
\begin{align*}
W(t)  =  \nabla_{X(t)}\left(v^+\right)
      = \frac{1}{\cos t}\cdot\nabla_{\hx(t)}\left((\mC\circ a)\cdot v^-\right)
      =  \frac{1}{\cos t} \cdot \left( d\mC_{a}(\tilde{X}(a)) + \mF^-_{\gamma(t)}(X)\cdot \mC(a) \right)\cdot v^-.
\end{align*}
Since $W(t)$ naturally extends smoothly beyond the value $t=\frac{\pi}{2}$, we see that $\lim_{t\ra\left(\frac{\pi}{2}\right)}|W(t)|$ is finite, from which the result follows.
\end{proof}

We will henceforth work entirely over $U^-$, which allows us to omit the ``$-$'' superscripts whenever it's convenient to do so, in particular defining $U=U^-$, $\omega=\omega^-$, $\mF=\mF^-$, and $p_0=p^-$. From this viewpoint, the results of this section can be summarized as follows.  We identify the bundle over $U=S^n-\{-p_0\}$ with $U\times\R^k$ via parallel transport along longitudes.  For each $v\in\R^k$, we denote by $v^-$ the corresponding ``constant'' section over $U$ with respect to this identification.  The connection over $U$ is fully determined by its connection form $\omega$, which is itself fully determined by the function $\mF$ that associated to each $p\in \tU=U-\{p_0\}$ a linear map $\mF_p:\mh\ra so(k)$.  More precisely, for every $v\in\R^k$, $X\in \mh$ and $p\in\tU$:
$$\nabla_{\ha} v^- = \omega(\ha)\cdot v^- = 0,\quad\text{ and }\quad (\nabla_{\hx} v^-)_p = \omega(\hx(p))\cdot v^-(p) = \mF_p(X)\cdot v^-(p).$$
Finally, the asymptotic behavior of $\mF$ is described as follows: as $p\ra -p_0$ along the longitude geodesic through $a\in S^{n-1}$, we have for all $X\in\mh$ that:
$$\mF_p(X) \ra -\mQ_a(X) = -d\mC_{a}(\tilde{X}(a))\cdot \mC(a)^{-1}.$$
\section{Examples} In this section, we describe the functions $\mC,\mF,\mQ$ for natural examples of parallel connections, beginning with the Levi-Civita connection.

\begin{lem} If $E=TS^n$, $\nabla$ is the Levi-Civita Connection, and $H=SO(n)$, then for all $p=((\cos\alpha)a,\sin\alpha)\in\tU$ and all $X\in\mh=so(n)$:
$$\mQ_a(X) = 2X^{\mpp_a},\text{ and } \mF_p(X) = -(\sin\alpha+1)X^{\mpp_a},$$
where the superscripts denotes the orthogonal projection onto the subspace $\mpp_a$.
\end{lem}

We will prove this using only elementary facts about parallel transport with respect to the Levi-Civita connection.

\begin{proof}
Let $a\in S^{n-1}$ and define $\beta_a$ as in Equation~\ref{PE}.  Assume $a\neq\pm a_0$ and let $a^\perp$ denote the normalized component of $a$ orthogonal to $a_0$, so that $a=(\cos\theta)a_0 + (\sin\theta)a^\perp$
for some $\theta\in(0,\pi)$.  A straightforward holonomy calculation in the totally geodesic $S^2$ formed by intersecting $S^n$ with $\text{span}\{p_0,a_0,a^\perp\}$ gives:
\begin{equation}\label{LC1}\mC(a) = P_{\beta_a} \text{ is a rotation by angle $2\theta$ in $\text{span}\{a_0,a^\perp\}$}.\end{equation}

We reinterpreted this conclusion in terms of the homogeneous description of the equatorial sphere as $S^{n-1}=H/K=SO(n)/SO(n-1)$ by considering the composition:
$$\mathfrak{p}\stackrel{\exp}{\longrightarrow} SO(n)\stackrel{\pi}{\longrightarrow} S^{n-1}\stackrel{\mC}{\longrightarrow}{SO(n)}.$$
For any $X\in\mathfrak{p}$, Equation~\ref{LC1} gives that $(\mC\circ\pi\circ\exp)(X)=\exp(2X)$.  Therefore, the image of the clutching map equals $\exp(\mpp)$, and furthermore the composition $\mC\circ\pi:SO(n)\ra SO(n)$ restricted to $\exp(\mathfrak{p})$ equals the squaring map $g\mapsto g^2$.

Using this characterization of $\mC$, we next compute $\mQ_{a_0}(X)$ for arbitrary $X\in\mpp$.  Since $\pi(I)=a_0$ and since $\mC\circ\pi$ equals the squaring map on $\exp(\mpp)$, we have:
$$\mQ_{a_0}(X) = d\mC_{a_0}(\tilde{X}(a_0)) =  d(\mC\circ\pi)_I(X)=\frac{d}{dt}\Big|_{t=0}\left(\exp(tX)\right)^2=2X.$$
Since $\mQ_{a_0}$ vanishes on $\mk$, we have more generally that $\mQ_{a_0}(X)=X^\mpp$ for all $X\in\mh$.  Since $a_0$ is only artificially distinguished in the discussion, it is straightforward to derive from this the more general rule that $\mQ_a(X)=X^{\mpp_a}$ for all $a\in S^{n-1}$ and all $X\in\mh$.

For $p=((\cos\alpha)a,\sin\alpha)\in\tU$ and $X\in\mpp_a$ with $|\tilde{X}(a)|=1$, we next wish to compute $\mF_p(X)$.  For this, define $\gamma_t$ as in Equation~\ref{beta}. A straightforward holonomy calculation in the totally geodesic $S^2$ formed by intersecting $S^n$ with $\text{span}\{p_0,a,\tilde{X}(a)\}$ gives that $P_{\gamma_t}$ equals a rotation by angle $t(\sin\alpha+1)$ in the plane spanned by $a$ and $\tilde{X}(a)$.  Therefore by definition of $\mF$:
\begin{align}\label{hihi}
\mF_p(X) = -\frac{d}{dt}\Big|_{t=0} P_{\gamma_t} & = -(\sin\alpha+1) \cdot\left(\text{infinitesimal rotation in $\text{span}\{a,\tilde{X}(a)\}$} \right)\\
  &= -(\sin\alpha+1)X.\notag
\end{align}
More generally, $\mF_p(X) = -(\sin\alpha+1)X^{\mpp_a}$ for every $X\in\mh$.
\end{proof}

\begin{prop}\label{P:e2} If $G=SO(n+1)$, $H=SO(n)$, $\rho:H\ra SO(k)$ is a homomorphism,
$E$ is the associated bundle $G\times_\rho\R^k$, and $\nabla$ be the connection inherited from the natural principal connection in the bundle $G\ra G/H$, then for all $p=((\cos\alpha)a,\sin\alpha)\in\tU$ and all $X\in\mh=so(n)$:
$$\mQ_a(X) = 2\varphi(X^{\mpp_a}),\text{ and } \mF_p(X) = -(\sin\alpha+1)\varphi(X^{\mpp_a}),$$
where $\varphi=d\rho_I:\mh\ra so(k)$.  This implies that for all $a\in S^{n-1}$ and $X\in\mpp_a$, we have:
\begin{equation}\label{E:tt}2(\tilde{X}Q)_a(Y)=-2(\tilde{Y}Q)_a(X)=[\mQ_a(X),\mQ_a(Y)] - \mQ_a([X,Y]).\end{equation}
\end{prop}
Notice that the previous Lemma is the special case of this proposition with $k=n$ and $\rho=I$.
\begin{proof}
According to~\cite[Diagram 1.2]{GSW}, parallel transport around a loop $\gamma$ in $S^n$ with respect to $\nabla$ (denoted $P_\gamma$) and with respect to the Levi-Civita connection (denoted $P_\gamma^{LC}$) are related by:
$$P_{\gamma} = \rho(P^{LC}_{\gamma}).$$
Let $a\in S^{n-1}$.  Defining $\beta_a$ as in Equation~\ref{LC1}, the clutching map of $\nabla$ (denoted $\mC$) is related to the clutching map of the Levi-Civita connection (denoted $\mC^{LC}$) as follows:
$$\mC(a) = P_{\beta_a} = \rho(P^{LC}_{\beta_a}) = \rho(\mC^{LC}(a)) = (\rho\circ\mC^{LC})(a),$$
in other words, $\mC = \rho\circ\mC^{LC}$.  Using ``LC'' superscripts for objects related to the Levi-Civita connection, we next relate $\mQ_a$ and $\mQ^{LC}_a$ by setting $y=\mC^{LC}(a)\in SO(n)$ and observing that for all $X\in\mh$:
\begin{align*}
  \mQ_a(X)
   & = d\mC_a(\tilde{X}(a))\cdot\mC(a)^{-1}
     = d(\rho\circ \mC^{LC})_a(\tilde{X}(a))\cdot(\rho\circ\mC^{LC})(a)^{-1} \\
   & = d\rho_y(d\mC^{LC}_a(\tilde{X}(a)))\cdot\rho(y)^{-1} = d\rho_y(\mQ_a^{LC}(X)\cdot y)\cdot\rho(y)^{-1}\\
   & = d\rho_I(\mQ_a^{LC}(X)) = d\rho_I(2X^{\mpp_a}) = 2\varphi(X^{\mpp_a}).
\end{align*}

To relate $\mF$ and $\mF^{LC}$, we define $\gamma_t$ as in Equation~\ref{hihi} and observe that:
$$\mF_p(X) = -\frac{d}{dt}\Big|_{t=0} P_{\gamma_t} = -\frac{d}{dt}\Big|_{t=0}\rho(P^{LC}_{\gamma_t})
  =d\rho_I(\mF_p^{LC}) = -(\sin\alpha+1)\varphi(X^{\mpp_a}).$$

Finally to prove Equation~\ref{E:tt}, let $a\in S^{n-1}$, $g\in\pi^{-1}(a)$, and $X,Y\in\mpp_a$.  Consider the path $g(t)=e^{tX}\cdot g$ in $H$ and the path $a(t) = \pi(g(t))$ in $S^{n-1}$.  Notice that $a(0)=a$ and $a'(0)=\tilde{X}(a)$. A straightforward calculation shows that $\frac{d}{dt}|_{t=0}Y^{\mpp_{a(t)}} = [X,Y]^{\mk_a}$, which is the same as $[X,Y]$ because $\mk\subset\mh$ is a symmetric pair.  Thus,
$$
(\tilde X\mQ)_a(Y)= \frac{d}{dt}\Big|_{t=0}\mQ_{a(t)}(Y)= 2\varphi\left(\frac{d}{dt}\Big|_{t=0}Y^{\mpp_{a(t)}}\right) = 2\varphi\left([X,Y]\right),$$
from which Equation~\ref{E:tt} follows.
\end{proof}

Our setup requires us to choose a Lie group $H$ that acts transitively on the equatorial sphere $S^{n-1}$.  We selected $H=SO(n)$ in the previous examples, but in certain situations it will be advantageous to choose a smaller-dimensional group.  For bundles over $S^4$, the choice $H=Sp(1)$ has the advantage of being a free action, so the isotropy group $K$ becomes trivial.  In particular, to study a quaternionic $\HH$-bundle over $S^4$ with a connection $\nabla$ that is compatible with the fiber quaternionic structure, we can consider $\mC:Sp(1)\ra Sp(1)$, where elements of the image act on the fiber $\R^4\cong\HH$ via quaternionic multiplication.  In the following well-known example, $\mC$ is the identity:

\begin{prop}If $S^3\ra S^7\ra\HP^1\cong S^4$ is the Hopf bundle, $E=(S^7\times\HH)/Sp(1)$ is the associated vector bundle, $\nabla$ is the connection inherited from the principal connection in the Hopf bundle, and $H=Sp(1)$, then $\mC:Sp(1)\ra Sp(1)$ is the identity map, and for all $p=((\cos\alpha)a,\sin\alpha)\in\tU$ and all $X\in\mh=sp(1)$:
$$\mQ_a(X) = X,\text{ and } \mF_p(X) = -\frac 12(\sin\alpha+1)X.$$
\end{prop}

This proposition can be verified by direct calculation, but this is not necessary because a minor rephrasing of the proposition is well-known.  The formula $\mF_p(X) = -\frac 12(\sin\alpha+1)X$ is equivalent to the following expression for the connection form:
$$\omega = \frac 12(\sin\alpha +1)(a\cdot d\overline{a}),$$
where the overline denotes quaternionic conjugation in $S^3\subset\R^4\cong\HH$.  The push-forward of $\omega$ via the stereographic projection map $P:U\ra\HH$ defined as $z=P((\cos\alpha)a,\sin\alpha) = \frac{(\cos\alpha)a}{1-\sin\alpha}$ equals:
$$P_*\omega = \frac{\text{Im}(z\cdot d\overline{z})}{1+z^2},$$
which is the standard expression in the mathematical physics literature for the basic instanton.  This expression first appeared in the literature as a surprising solution to the Yang-Mills equation on $\R^4$ (space-time), but its above-described relationship to the Hopf bundle eventually became well-understood; see for example~\cite{Atiyah}.

\section{Curvature formulas}
In this section, we return to the generality of Section 2, where $E$ is an arbitrary vector bundle over $S^n$ and $\nabla$ is an arbitrary connection described in local coordinates through the functions $\mC,\mF,\mQ$.  Our goal is to derive curvature formulas for $\nabla$ in terms of these functions.

Let $R^\nabla$ denote the curvature tensor of $\nabla$ and let $\Omega$ denote the $ so(k)$-valued two-form on $U$ defined such that for every pair $A,B$ of vector fields on $U$ and every $v\in\R^k$,
$$R^\nabla(A,B) v^- = \Omega(A,B)\cdot v^- .$$
The form $\Omega$ is determined by $\omega$ as follows:

\begin{lem}\label{L:contocurv} For any $X,Y\in sp(1)$ the following holds on $\tU$:
\begin{align*}
\Omega(\ha,\hx)   & = (\ha\mF)(X),\\
  \Omega(\hx,\hy) &= (\hx\mF)(Y) - (\hy\mF)(X) + \mF([X,Y]) + [\mF(X),\mF(Y)].
\end{align*}
\end{lem}

\begin{proof} It is straightforward to prove the following standard equation:
$$\Omega(A,B) = A(\omega(B))-B(\omega(A))-\omega([A,B]) + [\omega(A),\omega(B)]$$
for any pair $A,B$ of vector fields on $U$.  In particular:
\begin{align*}
\Omega(\ha,\hy) &= \ha\omega(\hy)-\hy\underbrace{\omega(\ha)}_0 - \omega(\underbrace{[\ha,\hy]}_0) + [\underbrace{\omega(\ha)}_0,\omega(\hy)] = (\ha \mF)(Y). \\
\Omega(\hat X,\hat Y) & = \hx\omega(\hy) - \hy\omega(\hx)-\omega([\hx,\hy]) + [\omega(\hx),\omega(\hy)] \\
   & = (\hx \mF)(Y) - (\hx \mF)(X) + \omega(\widehat{[X,Y]}) + [\mF(X),\mF(Y)].
\end{align*}
\end{proof}

We next collect some elementary facts about the Levi-Civita connection on $S^n$, which is henceforth denoted by ``$\overline\nabla$''.  First, we have for all $X,Y\in \mh$:
\begin{equation}\label{hello}
  [\hat\alpha,\hat X] = 0,\quad
  [\hat X,\hat Y]=-\widehat{[X,Y]},\quad\overline\nabla_{\hat\alpha}\hat\alpha=0,  \quad \overline\nabla_{\hat X}\hat\alpha = \overline\nabla_{\hat\alpha}\hat X = -(\tan\alpha)\hat X.
\end{equation}
Letting $\mathcal{S}$ denote shape operator of the latitude $S^{n-1}$ (the $\alpha$-level set), we next have:
$$\lb\overline\nabla_{\hx}\hy,\ha\rb = \lb \mathcal{S}_{\ha}(\hx),\hy\rb =- \lb\overline\nabla_{\hx}\ha,\hy\rb = (\tan\alpha)\lb\hx,\hy\rb.$$
The component of $\overline\nabla_{\hx}\hy$ tangent to the $\alpha$-level set at a point $p=((\cos\alpha)a,\sin\alpha)\in\tU$ is easier to express under the added assumption that $X,Y\in\mathfrak{p}_a$, in which case:
\begin{equation}\label{hello2} \left(\overline\nabla_{\hat X}\hat Y\right)_p = -\frac 12\widehat{[X,Y]}_p+(\tan\alpha)\lb \hx,\hy\rb_p\hat\alpha(p).
\end{equation}

We will use ``$D$'' to denote the covariant derivative of a tensor with respect to $\nabla$ and/or $\overline\nabla$.
Let $\Theta$ denote the $ so(k)$-valued tensor of order $3$ on $U$ defined as:
\begin{equation}\label{E:defTheta}\Theta(C,A,B) = (D_C\Omega)(A,B) + [\omega(C),\Omega(A,B)].\end{equation}
This tensor represents $DR^\nabla$ in the sense that:
\begin{lem}\label{L:dcurv} For every vector fields $A,B,C$ on $U$ and every $v\in\R^k$,
$$(D_CR^\nabla)(A,B)v^- = \Theta^\pm(C,A,B)\cdot v^-.$$
\end{lem}
\begin{proof} Let $\gamma(t)$ denote a path in $S^n$ with $\gamma(0)=p$ and $\gamma'(0)=C(p)$. Let $V_t=v^-(\gamma(t))$ and let $A_t,B_t$ denote the parallel extensions of $A,B$ along $\gamma$. Using the prime symbol to denote the $t=0$ value of the $\nabla$-covariant derivatives of a section along $\gamma$, we have at $p$:
\begin{eqnarray*}
(D_CR^\nabla)(A,B)V & = & (R^\nabla(A_t,B_t)V_t)' - R^\nabla(A,B)(V') \\
   & = & (\Omega(A_t,B_t)\cdot V_t)' - R^\nabla(A,B)(\omega(C)\cdot V) \\
   & = & (D_C\Omega)(A,B)\cdot V + \omega(C)\cdot\Omega(A,B)\cdot V -\Omega(A,B)\cdot\omega(C)\cdot V
\end{eqnarray*}
In the last equality, the fact that $(\Omega(A_t,B_t)\cdot V_t)'= (D_C\Omega)(A,B)\cdot V + \omega(C)\cdot \Omega(A,B)\cdot V$ follows from Lemma~\ref{L:lieb}.
\end{proof}

\section{Radially symmetric connections}
In this section, we prove the main theorem.  In our local coordinates, the hypothesis becomes: $\Theta(\ha,\ha,\hx)_p=0$ for all $p\in\tU$ and $X\in\mh$, or more concisely that $\Theta(\ha,\ha,\cdot)=0$ on $\tU$.  We first show that this hypothesis uniquely determine the connection in terms of $\mC$.
\begin{lem}\label{L:step1} If $\Theta(\ha,\ha,\cdot)=0$ on $\tU$, then for every $p=((\cos\alpha)a,\sin\alpha)\in \tU$ and every $X\in \mh$,
\begin{equation}\label{hereitis}\mF_p(X) = \underbrace{-\frac 12(\sin\alpha +1)}_{\text{denoted }g(\alpha)}\mQ_a(X).\end{equation}
\end{lem}
\begin{proof}
Using Lemma~\ref{L:contocurv} and Equation~\ref{hello}, we compute:
\begin{align*}
\Theta(\ha,\ha,\hx)
   & = (D_{\ha}\Omega)(\ha,\hx) + [\underbrace{\omega(\ha)}_0,\Omega(\ha,\hx)]\\
   & = \ha(\Omega(\ha,\hx)) - \Omega(\underbrace{\overline\nabla_{\ha}\ha}_0,\hx) - \Omega(\ha,\overline\nabla_{\ha}\hx)
     = (\ha\ha \mF)(X) + (\tan\alpha)\Omega(\ha,\hx) \\
   & = (\ha\ha \mF)(X) + (\tan\alpha)(\ha \mF)(X).
\end{align*}

Now fix $a\in S^{n-1}$ and consider the path $\gamma(\alpha)=((\cos\alpha) a,\sin\alpha)$ in $S^n$.  Denote $\mF_{\gamma(\alpha)}$ simply as $\mF(\alpha)$; that is, along $\gamma$ we are considering $\mF$ as a one-parameter family of linear maps from $\mh$ to $ so(k)$ with parameter $\alpha\in(-\pi/2,\pi/2)$.  The previous equation becomes:
\begin{equation}\label{E:spring}\mF''(\alpha)=-(\tan\alpha)\mF'(\alpha).\end{equation}
At the endpoints of the domain, we have for all $X\in \mh$:
$$\lim_{\alpha\ra-\pi/2}\mF(\alpha)= 0,\quad \lim_{\alpha\ra-\pi/2}\mF'(\alpha) = 0,
\quad \lim_{\alpha\ra\pi/2}\mF(\alpha)(X) =  -\mQ_a(X).$$
The first limit follows from $\lim_{\alpha\ra-\pi/2}|\hx|=0$.  The second limit comes from the fact that Equation~\ref{E:spring} remains bounded.  The third limit comes from Proposition~\ref{Mags}.  The proposition gives the unique solution to the differential equation with these initial and ending conditions.
\end{proof}

The previous lemma says that a connection satisfying $\Theta(\ha,\ha,\cdot)=0$ is determined in a canonical way by the clutching map.  But not every smooth map $\mC:S^{n-1}\ra SO(k)$ can be plugged into Equation~\ref{hereitis} to determine a smooth connection.  As the next lemma shows, smoothness of the connection at the poles $\pm p_0$ implies restrictions on $\mC$.

\begin{lem}\label{L:step2} If $\Theta(\ha,\ha,\cdot)=0$ on $\tU$, so that $\mF=g\cdot\mQ$ as in Equation~\ref{hereitis}, then for every $a\in S^{n-1}$ and every $X,Y\in\mathfrak{p}_a$, we have:
\begin{equation}\label{E:restriction}2(\tilde{X}Q)_a(Y)=-2(\tilde{Y}Q)_a(X)=[\mQ_a(X),\mQ_a(Y)] - \mQ_a([X,Y]).\end{equation}
\end{lem}
Notice that Equation~\ref{E:restriction} is the same as Equation~\ref{E:tt}, so it's already been verified for the examples of the previous section.
\begin{proof}
Fix $a\in S^{n-1}$ and unit-length $X,Y\in\mathfrak{p}_a$. Along the path $p(\alpha)=((\cos\alpha)a,\sin\alpha)$, consider the following function of $\alpha$, in which everything is assumed to be evaluated at $p(\alpha)$ or $a$ as appropriate:
\begin{align*}
\Theta(\hx,\hy,\ha)_{p(\alpha)}
  &= \hx(\Omega(\hy,\ha)) + [\omega(\hx),\Omega(\hy,\ha)] - \Omega(\overline{\nabla}_{\hx}\hy,\ha)
     - \Omega(\hy,\overline{\nabla}_{\hx}\ha) \\
  &= -g'(\alpha)(\tilde X \mQ)(Y)+ [g(\alpha)\mQ(X),-g'(\alpha)\mQ(Y)]
     +\frac 12\Omega(\widehat{[X,Y]},\ha)+(\tan\alpha)\Omega(\hy,\hx)\\
  & = -g'(\alpha)(\tilde X \mQ)(Y)- g(\alpha)g'(\alpha)[\mQ(X),\mQ(Y)]-\frac 12 g'(\alpha)\mQ([X,Y])\\
  &\quad -\tan(\alpha)\left( g(\alpha)(\tilde X \mQ)(Y)-g(\alpha)(\tilde Y \mQ)(X)+g(\alpha)\mQ([X,Y])
      +g(\alpha)^2[\mQ(X),\mQ(Y)]\right).
\end{align*}
Since $|\hx|=|\hy|=\cos\alpha$ at $p$, this yields:
\begin{equation}\label{E:cello} T(\alpha)=\Theta\left(\frac{\hx}{|\hx|},\frac{\hy}{|\hy|},\ha\right)_{p(\alpha)} = C_1 W_1 + C_2 W_2 + C_3 W_3,\end{equation}
where:
\begin{align*}
C_1  & =  \frac{-\cos^2\alpha + 2\sin(\alpha) + 2}{4\cos^3(\alpha)} ,
& C_2  & =  \frac{\sin\alpha+1}{2\cos^3(\alpha)},
& C_3  & =  \frac{-\sin\alpha\cdot(\sin\alpha + 1)}{2\cos^3\alpha},\\
W_1 & = \mQ([X,Y]) - [\mQ(X),\mQ(Y)],
& W_2 & = (\tilde{X}\mQ)(Y),
& W_3 & =(\tilde{Y}\mQ)(X).\\
\end{align*}

The smoothness of $\nabla$ at the poles $\pm p_0$ implies that $|T(\alpha)|$ remains bounded as $\alpha\ra\pm\pi/2$, which will imply relationships between $W_1$, $W_2$ and $W_3$.  More specifically, first notice that as $\lim_{\alpha\ra-\pi/2^+} C_1=0$ while $\lim_{\alpha\ra-\pi/2^+} C_2=\lim_{\alpha\ra-\pi/2^+} C_3=\infty$.
This implies that $W_2$ is parallel to $W_3$.  In fact, since $\lim_{\alpha\ra-\pi/2^+}\frac{C_2}{C_3}=1$, we learn that $W_2=-W_3$, so Equation~\ref{E:cello} becomes:
$$T(\alpha)=C_1W_1 + (C_2-C_3) W_2 \underbrace{= C_1(W_1+2W_2)}_{\text{because }2C_1=C_2-C_3}.$$
Since $\lim_{\alpha\ra \pi/2^-} C_1=\infty$, we learn that $W_1+2W_2=0$.  In summary, we have learned that $W_2=-W_3$ and $W_1+2W_2=0$, which is exactly Equation~\ref{E:restriction}.
\end{proof}

\begin{prop}\label{ups} If $\Theta(\ha,\ha,\cdot)=0$ on $\tU$, then $\nabla$ is parallel.
\end{prop}
\begin{proof}

Write $\mF=g\cdot\mQ$ as in Lemma~\ref{L:step1}.  Fix $p=((\cos\alpha)a,\sin\alpha)\in\tU$ and fix $X,Y,Z\in\mathfrak{p}_a$.
Equation~\ref{E:restriction} and Lemma~\ref{L:contocurv} now give that:
\begin{align*}
\Omega(\ha,\hx)_p & = g'\mQ_a(X),\\
\Omega(\hx,\hy)_p & = g(\tilde{X}\mQ)_a(Y) - g(\tilde{Y}\mQ)_a(X) + g \mQ_a([X,Y]) + g^2[\mQ_a(X),\mQ_a(Y)]\\
                & = (g^2+g)[\mQ_a(X),\mQ_a(Y)],
\end{align*}
where expressions involving $g$ are assumed to be evaluated at $\alpha$.

We need to show that $\Theta=0$.  Since $\{p,X,Y,Z\}$ were arbitrarily chosen, this is equivalent to showing:
\begin{equation}\label{round}\Theta(\ha,\ha,\hx)_p = \Theta(\ha,\hx,\hy)_p=\Theta(\hx,\hy,\ha)_p = \Theta(\hx,\hy,\hz)_p = 0.\end{equation}
First, $\Theta(\ha,\ha,\hx)=0$ by hypothesis.  Next:
\begin{align*}
\Theta(\ha,\hx,\hy)_p &= \ha\Omega(\hx,\hy)_p - \Omega(\overline{\nabla}_{\ha}\hx,\hy)_p
   - \Omega(\hx,\overline{\nabla}_{\ha}\hy)_p +[\underbrace{\omega(\ha)_p}_{0},\Omega(\hx,\hy)_p]\\
  & = (g^2+g)'[\mQ_a(X),\mQ_a(Y)] + 2(\tan\alpha)\Omega(\hx,\hy)_p \\
  & = \underbrace{\left(2gg'+g'+2(\tan\alpha)(g^2+g)\right)}_0[\mQ_a(X),\mQ_a(Y)]=0.
\end{align*}
The proof Lemma~\ref{L:step2} established that $\Theta(\hx,\hy,\ha)_p=0$.  So to verify Equation~\ref{round}, it remains to prove that the following vanishes:
$$\Theta(\hz,\hx,\hy)_p = \left(\hz\Omega(\hx,\hy)\right)_p - \Omega(\overline{\nabla}_{\hz}\hx,\hy)_p
                        - \Omega(\hx,\overline{\nabla}_{\hz}\hy)_p+[\omega(\hz)_p,\Omega(\hx,\hy)_p].$$
For this, we must first use Equation~\ref{E:restriction} to simplify the expression $\hz[\mQ(X),\mQ(Y)]$:
\begin{align*}
2\hz[\mQ(X),\mQ(Y)] &= 2[(\tilde{Z}\mQ)(X),\mQ(Y)] + 2[\mQ(X),(\tilde{Z}\mQ)(Y)] \\
   & =  [[\mQ(Z),\mQ(X)]-\mQ([Z,X]),\mQ(Y)] + [\mQ(X),[\mQ(Z),\mQ(Y)] - \mQ([Z,Y])]\\
   & = -[\mQ([Z,X]),\mQ(Y)] -[\mQ(X),\mQ([Z,Y])]\\
   & \qquad   +\Big(\underbrace{[[\mQ(Z),\mQ(X)],\mQ(Y)] + [\mQ(X),[\mQ(Z),\mQ(Y)]]}_{=-[[\mQ(X),\mQ(Y)],\mQ(Z)]\text{ by Jacobi}} \Big)\\
   & = -[\mQ([Z,X]),\mQ(Y)] -[\mQ(X),\mQ([Z,Y])] - [[\mQ(X),\mQ(Y)],\mQ(Z)].
\end{align*}
Next observe that $\lb\hz,\hy\rb_p=\cos^2(\alpha)\lb Z,Y\rb$ because the vectors lie in $\mathfrak{p}_a$ by hypothesis, so:
\begin{align*}
\Omega(\overline{\nabla}_{\hz}\hx,\hy)
    & = \Omega\left(-\frac 12\widehat{[Z,X]}+(\tan\alpha)\lb \hz,\hx\rb_p\hat\alpha,\hy\right) \\
    & = -\frac 12(g^2+g)[\mQ([Z,X]),\mQ(Y)]+g'(\cos\alpha)(\sin\alpha)\lb Z,X\rb\mQ(Y),
\end{align*}
And similarly for $\Omega(\hx,\overline{\nabla}_{\hz}\hy)$, so that:
\begin{align*}
  -\Omega(\overline{\nabla}_{\hz}\hx,\hy)-\Omega(\hx,\overline{\nabla}_{\hz}\hy)
  & = \frac 12(g^2+g)\left([\mQ([Z,X]),\mQ(Y)] + [\mQ(X),\mQ([Z,Y])]\right) \\
  &\qquad -g'(\cos\alpha)(\sin\alpha)\left(\lb Z,X\rb \mQ(Y) - \lb Z,Y\rb \mQ(X) \right).
\end{align*}
Finally,
$$[\omega(\hz),\Omega(\hx,\hy)] = g(g^2+g)[\mQ(Z),[\mQ(X),\mQ(Y)]].$$
These expressions give the following simplification:
\begin{align*}
\Theta(\hz,\hx,\hy) &= \underbrace{\hz\Omega(\hx,\hy)}_{(g^2+g)\tilde{Z}[\mQ(X),\mQ(Y)]} +[\omega(\hz),\Omega(\hx,\hy)] - \Omega(\overline{\nabla}_{\hz}\hx,\hy)
                        - \Omega(\hx,\overline{\nabla}_{\hz}\hy)\\
                   &=  -(g^2+g)(g+1/2)[[\mQ(X),\mQ(Y)],\mQ(Z)]\\
                    & \qquad  -g'(\cos\alpha)(\sin\alpha)\left(\lb Z,X\rb \mQ(Y) - \lb Z,Y\rb \mQ(X) \right)\\
                    & = \frac 18\cos^2(\alpha)\sin(\alpha)
                        \Big( 4\lb Z,X\rb \mQ(Y) - 4 \lb Z,Y\rb \mQ(X) - [[\mQ(X),\mQ(Y)],\mQ(Z)]  \Big).
\end{align*}
Like in the proof Lemma~\ref{L:step2}, the smoothness of $\nabla$ at the poles $\pm p_0$ implies that the function $T(\alpha) = \Theta\left(\frac{\hz}{|\hz|},\frac{\hx}{|\hx|},\frac{\hy}{|\hy|}\right)$ has bounded norm as $\alpha\ra\pm\pi/2$, which implies that:
\begin{equation}\label{E:restriction2} [[\mQ(X),\mQ(Y)],\mQ(Z)] = 4\lb Z,X\rb \mQ(Y) - 4 \lb Z,Y\rb \mQ(X).\end{equation}
Thus, $\Theta(\hz,\hx,\hy)=0$ as desired.

This completes the proof, although it's worth mentioning that Equation~\ref{E:restriction2} is true for the associated bundles described in Proposition~\ref{P:e2}, simply because $\varphi$ is a Lie algebra homomorphism and $H/K=S^{n-1}$ is a symmetric space with constant curvature $1$.
\end{proof}

Given a rank $k$ vector bundle over $S^n$ and an explicit map $\mC:S^{n-1}\ra SO(k)$ with the correct homotopy type to represent the clutching map of the bundle, one might wish that a connection in the bundle could be canonically determined from $\mC$, but the above proof suggests not.  Equation~\ref{hereitis} prescribes a connection in terms of $\mC$, and is the only prescription that matches the natural examples, but yet the proof shows that the resulting connection is smooth only for the known examples.

In contrast, an explicit representative of the homotopy type of the \emph{classifying} map is known to canonically determine a connection (see~\cite{NR}), but is often more difficult to work with than the clutching map.  For example, the parallel and radial symmetry hypotheses have been translated into geometric conditions on the classifying map (see~\cite{T2}), but this translation probably doesn't yield any natural proof the the main theorem of this paper.


\bibliographystyle{amsplain}

\end{document}